\documentclass[11pt]{amsart}
\usepackage{amssymb}
\usepackage{verbatim}
\usepackage{amscd}
\usepackage{amsmath}

\theoremstyle{plain}
\newtheorem{theorem}{Theorem}[section]
\newtheorem{corollary}[theorem]{Corollary}
\newtheorem{lemma}[theorem]{Lemma}
\newtheorem{proposition}[theorem]{Proposition}

\newtheorem{remark}[theorem]{Remark}

\numberwithin{equation}{section}

\def\boxit#1{\vbox{\hrule\hbox{\vrule\kern3pt
     \vbox{\kern3pt#1\kern3pt}\kern3pt\vrule}\hrule}}

\newfont{\msam}{msam10}            
\newfont{\msym}{msbm10 scaled\magstep1}            

\newfont{\gotic}{eufm10 scaled\magstep1}

\newcommand{\ra}{\rightarrow}
\newcommand{\lra}{\mbox{\Huge $\longrightarrow$}}

\newcommand{\kra}{\kern-7pt\rightarrow\kern-7pt}

\def\bbh{{\mathbb H}}

\def\bbr{{\mathbb R}}

\def\bbn{{\mathbb N}}
\def\bbc{{\mathbb C}}

\def\ra{\rightarrow}

\def\lra{\longrightarrow}

\numberwithin{equation}{section}

\def\begtab{\begin{tabbing} WW\=23/02: \= point 1\kill}

\def\NIL1{{\mathcal H^{3}}}

\def\wt{\widetilde}

\def\n2{\mathfrak{N}_2}

\def\frakm{{\mathfrak m}}

\begin{document}

\title[ Area and holonomy  on   the principal $U(n)$ bundles ]{  Area and holonomy   on   the principal $U(n)$ bundles over the dual of Grassmannian manifolds}

\author{Taechang Byun}
\address{Department of Mathematics and Statistics, Sejong University,
Seoul 143-747, Korea}
\email{tcbyun@gmail.com}



\maketitle

\begin{abstract}
Consider the principal $U(n)$ bundles over the dual of Grassmann manifolds
$U(n)\ra U(n,m)/U(m) \stackrel{\pi}\ra D_{n,m}$.
Given a 2-dimensional subspace 
$\frakm' \subset \frakm $ $ \subset \mathfrak{u}(n,m), $ 
assume 
either 
$\frakm'$ is induced by $X,Y \in U_{m,n}(\bbc)$
with $X^{*}Y = \mu I_n$ for some $\mu \in \bbr$
or
by $X,iX \in U_{m,n}(\bbc)$.  
Then $\frakm'$  gives rise to a complete totally geodesic surface $S$ in the base
space. Furthermore,
let $\gamma$
be a piecewise smooth, simple closed curve on $S$
parametrized by $0\leq t\leq 1$, and $\wt\gamma$ its horizontal lift on
the bundle
$U(n) \ra \pi^{-1}(S) \stackrel{\pi}{\rightarrow} S,$
which is immersed in
$U(n) \ra U(n,m)/U(m) \stackrel{\pi}\ra D_{n,m} $.
Then
$$
\wt\gamma(1)= \wt\gamma(0) \cdot ( e^{i \theta} I_n)
\text{\hskip24pt or\hskip12pt } \wt\gamma(1)= \wt\gamma(0),
$$
depending on whether $S$ is a complex submanifold or not,
where $A(\gamma)$ is the area of the region on the surface $S$
surrounded by $\gamma$ and 
$\theta= 2  \cdot \tfrac{1}{n} A(\gamma).$
\end{abstract}

\section{Introduction}

For two natural numbers $n,m \in \bbn,$ 
define an Hermitian form $F : \bbc ^{n+m} \rightarrow \bbc$  by
\begin{align*}
  F(v,w) & = v^* \, \Lambda^n _{m} \, w \\
         & = -\sum_{k=1}^{n} \bar{v}_k w_k + \sum_{s=n+1}^{n+m}\bar{v}_s w_s,
\end{align*}         
where $v$ and  $w$ are regarded as column vectors, and
$$
  \Lambda^{n} _{m} =
  \left(
        \begin{array}{cccc}
          -I_n & O_{n \times m} \\
          O_{m \times n} & I_m \\
        \end{array}
  \right).
$$
Consider the Lie group
$$
  U(n,m) 
  = 
  \{ \Phi \in GL_{n+m}(\bbc) | 
         \Phi^* \: \Lambda^{n} _{m} \: \Phi \: =  \: \Lambda^{n} _{m}
  \}.
$$
Then, 
$$
  U(n,m) 
  = 
  \{ \Phi \in GL_{n+m}(\bbc) \, \big| \, 
         F(\Phi v, \Phi w ) = F(v,w), \, v,w \in \bbc ^{n+m}
  \},
$$
and
$D_{n,m} := U(n,m)/\left(U(n) \times U(m)\right)$ can be identified 
with the set of all $n$-dimensional subspaces $V$ of $\bbc ^{n+m}$ such that 
$F(v,v) \le 0$ for every $v \in V$
by considering the first $n$ columns of an element in $U(n,m).$

Let
$$
  U_{m,n}(\bbc):= 
  \{ 
     X \in M_{m,n}(\bbc) 
     \,\, | \,\, 
     X^{*} X = \lambda I_n \text{ for some } \lambda \in \bbc - \{0\}
  \},
$$
which may be regarded as generalizations of a unitary group.
As it did for  
the princiapl $U(n)$ bundles 
$$U(n) \ra U(n+m)/U(m) \ra G_{n,m}$$ over 
a Grassmannain manifold $G_{n,m} $ in \cite{BC},
it will also play an important role in studying 
the princiapl $U(n)$ bundles 
$$U(n) \ra U(n,m)/U(m) \ra D_{n,m}$$ 
over $D_{n,m}$
such that 
$U(n,m)$ has a left invariant metric, related to the Killing-Cartan form, given by
$$
  \langle A,B \rangle 
   = \tfrac{1}{2} \text{Re}\big(\text{Tr}(A^{*}B)\big), 
      \qquad A,B \in \mathfrak{u}(n,m),
$$
and that 
$D_{n,m} = U(n,m)/\left(U(n) \times U(m)\right)$
has the induced metric which makes the projection a Riemannian submersion.

\bigskip
Consider the Hopf fibration $S^1\ra S^3\ra S^2$. Let $\gamma$ be a
simple closed curve on $S^2$. Pick a point in $S^3$ over
$\gamma(0)$, and take the unique horizontal lift $\wt\gamma$ of
$\gamma$. Since $\gamma(1)=\gamma(0)$, $\wt\gamma(1)$ lies in the
same fiber as $\wt\gamma(0)$ does. We are interested in
understanding the difference between $\wt\gamma(0)$ and
$\wt\gamma(1)$. The following equality was already known
\cite{Pin}:
 $$
V(\gamma)=e^{\frac{1}{2}  A(\gamma) i},
$$
where $V(\gamma)$ is  the holonomy
displacement along $\gamma$, and
$A(\gamma)$ is the area of the region surrounded by $\gamma$.

\medskip
In this paper, we  generalize this fact to the $U(n)$ bundle over 
$D_{n,m}$ through $U_{m,n}(\bbc)$
$$
   U(n) \ra U(n,m)/U(m) \stackrel{\pi}\ra D_{n,m}.
$$
\noindent
The main results are stated as follows:
For $\hat{X} \in {\mathfrak u}(n,m)$ which is induced by
$X \in U_{m,n}(\bbc)$, 
consider a 2-dimensional subspace 
$\frakm' \subset \frakm  $ $ \subset \mathfrak{u}(n,m) $ 
with $\hat{X} \in \frakm'.$
Assume either 
$$\frakm' = \text{Span}_{\bbr}\{\hat{X},\hat{Y}\}$$
for some $Y \in U_{m,n}(\bbc)$ with $X^{*}Y = \mu I_n$ for some $\mu \in \bbr$
or
$$\frakm' = \text{Span}_{\bbr} \{\hat{X},\widehat{iX}\}.$$
Then $\frakm'$  gives rise to a complete totally geodesic surface $S$ in the base
space. Furthermore,
let $\gamma$
be a piecewise smooth, simple closed curve on $S$
parametrized by $0\leq t\leq 1$, and $\wt\gamma$ its horizontal lift on
the bundle
$U(n) \ra \pi^{-1}(S) \stackrel{\pi}{\rightarrow} S,$
which is immersed in
$U(n) \ra U(n,m)/U(m) \stackrel{\pi}\ra D_{n,m} $.
Then
$$
\wt\gamma(1)= \wt\gamma(0) \cdot ( e^{i \theta} I_n)
\text{\hskip24pt or\hskip12pt } \wt\gamma(1)= \wt\gamma(0),
$$
depending on whether 
$S$ is a complex submanifold or not,
where $A(\gamma)$ is the area of the region on the surface $S$
surrounded by $\gamma$ and 
$\theta= 2 \cdot \tfrac{1}{n} A(\gamma).$
See {\rm Theorem  \ref{thm-sphere}}
\bigskip

\section{ Preliminaries }
\label{pre}

\bigskip
To begin with, we introduce the results of \cite{CL}:
for the circle group
$$
  S^1 \cong S(U(1) \times U(1))=
  \left\{
        \left(
              \begin{array}{cccc}
                e^{-iz} & 0 \\
                O      & e^{iz} \\
              \end{array}
        \right)
        \:  : \:
        0 \le z \le 2 \pi
  \right\},
$$
consider the principal bundle
$$S^1 \lra SU(1,1) \stackrel{p}\lra \bbc H^1$$
under the representation of $SU(1,1)$ into $\text{GL}(4,\bbr)$ such that
$$
  w =   \left(
              \begin{array}{rrrr}
                 w_1 &  w_2 &  w_3 &  w_4 \\
                -w_2 &  w_1 & -w_4 &  w_3 \\
                 w_3 & -w_4 &  w_1 & -w_2 \\
                 w_4 &  w_3 &  w_2 &  w_1 \\
              \end{array}
        \right)
$$ 
with the condition $- w_1 ^2 - w_2 ^2 + w_3 ^2 + w_4 ^2 = -1,$
which induces the following identifications
\begin{align*}
  \bbc H^1 
  &=
  \left\{
        \left(
              \begin{array}{rrrr}
                 x &  0 &  y &  z \\
                 0 &  x & -z &  y \\
                 y & -z &  x &  0 \\
                 z &  y &  0 &  x \\
              \end{array}
        \right)
        \:  : \:
        -x^2 + y^2 + z^2 = -1, \, x>0
  \right\}  \\
  &=
  \left\{ (x,y,z) \in \bbr^3 \; : \;  -x^2 + y^2 + z^2 = -1, \, x>0  \right\} \\
  &=: \bbh^2,  
\end{align*}
where 
$$p: SU(1,1) \lra \bbc H^1$$is defined by $p(w) = w \tilde{w}$
with
$$
  \tilde{w} 
  =   \left(
              \begin{array}{rrrr}
                 w_1 & -w_2 &  w_3 &  w_4 \\
                 w_2 &  w_1 & -w_4 &  w_3 \\
                 w_3 & -w_4 &  w_1 &  w_2 \\
                 w_4 &  w_3 & -w_2 &  w_1 \\
              \end{array}
        \right).
$$ 
Note that $p$ has the following properties:
\begin{align*}
p(wv)&=w p(v) \wt w\quad\text{for all } w,v\in SU(2)\\
p(wv)&=p(w) \quad\text{if and only if}\quad 
              v\in S\big(U(1) \times U(1)) \cong S^1,
\end{align*}
which shows that it is the orbit map of the principal bundle. 
But we have to be careful that the inclusion map 
$\bbc H^1 \hookrightarrow SU(1,1)$ 
is not a cross-section in this bundle. In fact,
$p(v) = v^2 \in \bbc H^1$ for any $ v \in \bbc H^1.$

\begin{theorem}[\cite{CL}]
\label{area-u2}
Let 
$ 
  S^1\lra SU(1,1) \stackrel{p}\lra 
  \big(\bbc H^1, \langle \cdot, \cdot \rangle _{\bbh^2} \big)
$ 
be the natural fibration.
Let $\gamma$ be a piecewise smooth, simple closed curve on $\bbc H^1$.
Then the holonomy displacement along $\gamma$ is given by
$$
  V(\gamma)= e^{\frac{1}{2} A(\gamma) i}  
$$
where $A(\gamma)$ is the area of the region on $\bbc H^1$
enclosed by $\gamma.$ 
\end{theorem}

\bigskip

To apply the result of Theorem \ref{area-u2},
we will study the isomorphic equivalence of the principal bundle
$$S^1 \lra SU(1,1) \stackrel{p}\lra \bbc H^1$$
to the one
$$
  S\big( U(1) \times U(1) \big)
  \ra
  SU(1,1)
  \ra
  SU(1,1) / S\big(U(1) \times U(1)),
$$ 
but not the isometric equivalence. 
In fact, a conformal map $h: SU(1,1) / S\big(U(1) \times U(1)) \ra \bbc H^1$ will be constructed such that the identity map on $SU(1,1)$ is the bundle map covering  it.

\bigskip
Define a map $h: SU(1,1)/S\big(U(1) \times U(1)\big) \ra \bbc H^1$ by
$$h(v H) = v^2 = p(v) \qquad v \in \bbc H^1,$$
where $H = S\big(U(1) \times U(1)\big).$
Then, the identity map of $SU(1,1)$ is a trivially isomorphic bundle map which covers the map $h.$

\bigskip
The Lie group $SU(1,1)$ will have a left-invariant Riemannian metric
given by the
following orthonormal basis on the Lie algebra $\mathfrak{su}(1,1) $
$$
  E_1 =
  \left[
          \begin{array}{cc}
            0 & 1 \\ 1 & 0
          \end{array}
  \right]\ ,
  \quad
  E_2 =
  \left[
          \begin{array}{cc}
            0 & -i \\ i & 0
          \end{array}
  \right]\ ,
  \quad
  E_3 =
  \left[
          \begin{array}{cc}
            -i & 0 \\ 0 & i
          \end{array}
  \right]\ ,
$$
which correspond to
\[
e_{1}= \left( \begin{array}{rrrr}
0 & 0 & 1& 0  \\
0 & 0 & 0  & 1\\
1 & 0 & 0 &0\\
0 & 1 & 0 &0
\end{array}\right),\;\;
e_{2}= \left( \begin{array}{rrrr}
0 & 0 & 0& 1  \\
0 & 0 & -1  & 0\\
0 & -1 & 0 &0\\
1 & 0 & 0 &0
\end{array}\right) ,\;\;
e_{3}=
\left( \begin{array}{rrrr}
0 & 1 & 0& 0  \\
-1 & 0 & 0  & 0\\
0 & 0 & 0 &-1\\
0 & 0 & 1 &0\\
\end{array}\right) 
\]
in $\frak{gl}(2k,\bbr),$ respectively.
Notice that 
$[e_1,e_2]= -2 e_3$.

\bigskip
In order to understand 
the map $h$ between base spaces and the projection map $p$ better, 
refer to \cite{CL} and consider the subset of $SU(1,1)$:
\begin{align*}
T
&=
\left\{
\left[\begin{array}{ll}
\cosh x & (\sinh x)e^{-i y}\\
(\sinh x)e^{i y} &\cosh x\\
\end{array}\right]\ :\  x\ge 0,\ 0\leq y\leq 2\pi
\right\}\\
&=
\left\{
\left[\begin{array}{rrrrrrrr}
\cosh x &0          &(\sinh x)(\cos y) &(\sinh x)(\sin y)\\
0          &\cosh x & -(\sinh x)(\sin y) &(\sinh x)(\cos y)\\
 (\sinh x)(\cos y) &-(\sinh x)(\sin y) &\cosh x &0         \\
 (\sinh x)(\sin y) & (\sinh x)(\cos y) &0          &\cosh x\\
\end{array}\right]
\right\} \\
\end{align*}
which is the exponential image of
$$
\frakm=
\left\{\left[\begin{array}{cc}
 0 &  \bar{\xi}^{t} \\ \xi &0
\end{array} \right]\ :\ \xi  \in  {\bbc} \right\}.
$$
Furthermore, it is exactly same as $\bbc H^1,$ so the map $p$ restricted to $T$ is just the squaring map; that is,
$$
p(w)=w^2,\quad w\in T.
$$

\medskip
To check $h$ is a conformal map:
given 
$$
  w =\big(\cosh{x}, (\sinh{x})(\cos{y}), (\sinh{x})(\sin{y}) \big) \in T 
    = \bbc H^1,
$$
\begin{align*}
  \big| D_1(wH) \big| 
  &= \big|(D_1 w)^{\text{h}} \big| \\
  &= \big|
           \big(
                 (\cos{y}) {L_w}_{*} e_1 + (\sin{y}) {L_w}_{*} e_2
           \big)^{\text{h}} 
     \big| \\
  &= 1
\end{align*}
and
\begin{align*}
  &\big| D_2(wH) \big| = \big|(D_2 w)^{\text{h}} \big| \\
  &= \big| 
           \big( -\tfrac{1}{2} (\sinh{2x})(\sin{y}) {L_w}_{*} e_1 
                   + \tfrac{1}{2} (\sinh{2x})(\cos{y}) {L_w}_{*} e_2
                   + (\sinh^2 {x}) {L_w}_{*} e_3
           \big)^{\text{h}} 
     \big| \\
  &= \tfrac{1}{2} \big| \sinh{2x} \big|,
\end{align*}
while
\begin{align*}
  \big| D_1 \, h(wH) \big| 
  &= \big| D_1 \, w^2 \big| \\
  &= \big|
          ( 2\sinh{2x}, \, 2 (\cosh{2x}) (\cos{y}), \, 2 (\cosh{2x}) (\sin{y}) )
     \big| \\
  &= 2
\end{align*}
and
\begin{align*}
  \big| D_2 \, h(wH) \big| 
  &= \big| D_2 \, w^2 \big| \\
  &= \big|
             (    0, \, - (\sinh{2x}) (\sin{y}), \,  (\sinh{2x}) (\cos{y})  )
     \big| \\
  &= \big| \sinh{2x} \big|.
\end{align*}
Thus $h$ is a conformal map.

\bigskip
From Theorem \ref{area-u2}, we get the following result:

\begin{theorem} \label{simple}
Consider 
$$S(U(1) \times U(1)) \ra SU(1,1)\ra SU(1,1)/S(U(1) \times U(1)) .$$ 
Let $\gamma$ be a piecewise smooth, simple closed curve on 
$SU(1,1)/S(U(1) \times U(1)) $.
Then the holonomy displacement along $\gamma$ is given by
$$
  V(\gamma)= e^{2 A(\gamma)\Phi} = e^{\frac{1}{2} A(h \circ \gamma) i}  
  \in S(U(1) \times U(1))  \cong S^1
$$
where $A(\gamma)$ is the area of the region on $SU(1,1)/S(U(1) \times U(1))$
enclosed by $\gamma$ and 
$$
  \Phi = 
  \left[
          \begin{array}{cc}
            i & 0 \\ 0 & -i
          \end{array}
  \right] .
$$
\end{theorem}

\bigskip
\section{
             The bundle
             $U(n)\lra U(n,m)/U(m)\lra D_{n,m}$
           }

To deal with the bundle
$$U(n)\lra U(n,m)/U(m)\lra D_{n,m},$$
we investigate
$$U(n)\lra U(n,m) \lra D_{n,m}.$$

The Lie algebra $\mathfrak{u}(n,m)$ of $U(n,m)$ has the following canonical decomposition:
$$\mathfrak{u}(n,m) = \mathfrak{h} + \mathfrak{m},$$
where
$$
  \mathfrak{h}
  =\mathfrak{u}(n)+ \mathfrak{u}(m)
  =
  \left\{
        \left(
              \begin{array}{cccc}
                A & O_{n \times m} \\
                O_{m \times n} & B \\
              \end{array}
        \right)
        \:  : \:
        A \in \mathfrak{u}(n),\:  B \in \mathfrak{u}(m)  
  \right\}
$$
and
$$
  \mathfrak{m}
  =
  \left\{
        \hat{X} :=   
        \left(
              \begin{array}{cccc}
                O_{n} & X^{*} \\
                X     & O_{m}   \\
              \end{array}
        \right)
        \: : \:
        X \in M_{m,n}(\bbc)  
  \right\}.
$$

\bigskip
  Define an Hermitian inner product $h : \bbc ^m \rightarrow \bbc$ by
$$h(v,w) = v^*\, w,$$
where $v$ and  $w$ are regarded as column vectors.
Then the following lemma is obvious:

\begin{lemma}
\label{lambda}
If a matrix $X \in M_{m,n}(\bbc)$ satisfies $X^{*}X = \lambda I_n$ for some $\lambda \in \bbc,$ then $\lambda$ will be a nonnegtive real number and $\lambda=0$ only if $X$ is trivial. 
\end{lemma}

\bigskip
From the lemma \ref{lambda}, 
we obtain that
\begin{align*}
  U_{m,n}(\bbc)
  &= 
  \{ 
     X \in M_{m,n}(\bbc) 
     \,\, | \,\, 
     X^{*} X = \lambda I_n \text{ for some } \lambda \in \bbc - \{0\}
  \}
  \\
  &=
  \{ 
     X \in M_{m,n}(\bbc) 
     \,\, | \,\, 
     X^{*} X = \lambda I_n \text{ for some } \lambda > 0
  \}.
\end{align*}

\bigskip
\begin{lemma}
\label{calculation}
   Let
   $$
      X = \Big(a^r_k + i b^r_k \Big) \, ,
      Y = \Big(c^r_k + i d^r_k \Big)
      \in M_{m,n}(\bbc)$$
for  $r=1, \cdots, m$, and  $ k=1, \cdots, n$. Suppose that  for their induced
 $\hat{X}, \, \hat{Y} \in \frakm$,
 $$
    [[\hat{X}, \, \hat{Y}], \hat{X}]  = \hat{Z} \in \frakm
 $$
 for some
  $
   Z  =
   \Big(\alpha ^r_k \Big) \in M_{m,n}(\bbc)
  $ for $r=1, \cdots, m$, and  $k=1, \cdots, n$.
  Then we have
 $$
   \alpha ^r_k =
   \sum^{n}_{j=1} (a^r_j + i b^r_j) \big(\!2 h(Y_j, X_k) - h (X_j, Y_k)\big)
   - \sum^{n}_{j=1} (c^r_j + i d^r_j) \, h(X_j, X_k),
 $$
where $X_k$ and $Y_k$ are $k$-column vectors of $X$ and $Y$ for 
$k= 1, \cdots, n,$ respectively.
\end{lemma}

\begin{proof}
It is easily obtained from
 $$
    [[\hat{X}, \, \hat{Y}], \hat{X}]  
    = \hat{X} (2 \hat{Y}\hat{X} - \hat{X}\hat{Y}) - \hat{Y}\hat{X}\hat{X}. 
 $$
\end{proof}

\bigskip
  Recall the following proposition, which gives the clue for the holonomy displacement in the principal $U(n)$ bundles over $D_{n,m},$ 
$$U(n) \lra U(n,m)/U(m) \stackrel{\pi} \lra D_{n,m}.$$

\begin{proposition} \cite{KN} \label{affine}
 Let $(G,H,\sigma)$ be a symmetric space and
 $\mathfrak{g} = \mathfrak{h} + \mathfrak{m}$
 the canonical decomposition. Then there is a natural one-to-one correspondence between the set of linear subspaces $\mathfrak{m}'$ of $\mathfrak{m}$ such that
$[[\mathfrak{m}', \mathfrak{m}'], \mathfrak{m}'] \subset \mathfrak{m}'$
and the set of complete totally geodesic submanifolds $M'$ through the origin $0$ of the affine symmetric space $M=G/H,$ the correspondence being given by 
$\mathfrak{m}' = T_0 (M').$
\end{proposition}

Note that $\mathfrak{m}'$ in the Proposition \ref{affine} will make a bunch of complete totally geodesic submanifolds, each of which is obtained from another one by a translation, in the affine symmetric space $G/H.$

\bigskip
The role of $U_{m,n}(\bbc)$ in this paper will be seen from now on.

\bigskip

\begin{theorem} \label{easy}
Given 
$X \in U_{m,n}(\bbc)$ and
the natural fibration
$$
 U(n) \times U(m) \ra U(n,m) \ra D_{n,m},
$$
assume a 2-dimensional subspace 
$\frakm' = \text{Span}_{\bbr} \{ \hat{X}, \hat{Y}\} $ of
$\frakm\subset {\mathfrak u}(n,m)$
satisfies
\begin{align}
   X^* \, X = \lambda I_n, \quad  X^* \, Y = \mu I_n, \qquad 
    \mu \in \bbc \label{STAR}
\end{align}
for  $ Y \in M_{m,n}(\bbc).$
Then $\frakm'$ gives rise to a complete totally geodesic  surface $S$ in 
$D_{n,m}$ 
if and only if 
either
($ \text{Im}\,{\mu} \not= 0$ and $iX \in \text{Span}_{\bbr}\{ X, Y \}$)
or
($ \text{Im}\,{\mu} = 0 $ and $Y \in U_{m,n}(\bbc)$)
holds.
\end{theorem}

\begin{proof}
 To begin with, note that $\lambda >0.$
 Assume that  $\frakm'$ gives rise to a complete totally geodesic  surface $S$  in $D_{n,m}$.
By a translation, without loss of generality, we can assume that $S$ passes through the origin of the affine symmetric space 
$D_{n,m} = U(n,m)/\left(U(n) \times U(m)\right).$

 To show the necessary condition: let $e_k \in \bbc^m, \: k=1, \cdots, m,$ be an elementary vector which has all components 0 except for the $k$-component with 1. Then
  $$h(X_k, Y_j) = h(X e_k, Y e_j) = e_k ^* (X^* Y) e_j ,$$
  so the condition (\ref{STAR}) is equivalent to
  $$
     h(X_k, Y_k) = \mu, \quad
     h(X_k, X_k) = \lambda, \quad
     h(X_k, X_j) =0 , \quad
     h(X_k, Y_j) =0
  $$
  \noindent  
  for $k \not= j $ in $\{1, \cdots ,n \}$.  From $ h(X_k, Y_k) = \mu $,
 we obtain
  $$
    2h(Y_k, X_k) - h(X_k, Y_k) =
     \text{Re} \mu - 3i \text{Im}  \mu.
  $$
  \noindent
   The hypothesis of totally geodesic and Proposition \ref{affine} say that \\
$$ a \hat{X} + b \hat{Y} = [[\hat{X}, \hat{Y}], \hat{X}] $$
for some $a, b \in \bbr$. 
So, from Lemma ~\ref{calculation},  
  \begin{align*}
    a X + b Y
    &= (\text{Re}\mu - 3i \text{Im}\mu) X - \lambda Y
    \\
    &=  -3 \text{Im}\mu (i X)
          + (\text{Re}\mu  X - \lambda Y) ,        
  \end{align*}
and then $\text{Im}\mu \not= 0$ implies 
$i X$ will lie in $\text{Span}_{\bbr} \{ X, Y \} $ 
and that
$\widehat{iX}$ will lie in 
$\text{Span}_{\bbr} \{\hat{X}, \hat{Y} \} = \frakm' \subset \mathfrak{u}(n,m)$.

  If $\text{Im}\mu = 0$, then
  $$
    X^* Y -  Y^* X =  X^* Y - (X^* Y)^*
    =   2 i \text{Im}\mu \, I_n = O_n,
  $$ 
  \noindent 
  so
  \begin{align}
  [\hat{X}, \hat{Y}] =
  \left[
          \begin{array}{cc}
            O_n & 0 \\ 0 &  X Y^* - Y X^*
          \end{array}
  \right]\
  \in \mathfrak{u} (m) \subset \mathfrak{u} (n,m).
  \label{su_M}
  \end{align}
  Let $M =  X Y^* - Y X^*$. Then
  $$
  [\hat{X}, \hat{Y}] =
  \left[
          \begin{array}{cc}
            O_n & 0 \\ 0 & M       
          \end{array}
  \right]    
  $$
  and
  $ [[\hat{Y}, \hat{X}], \hat{Y}] = - \widehat{MY} \in \frakm'$
  from the hypothesis of the condition of totally geodesic and from
  Proposition \ref{affine}.
  Note that
  $$
     MY =  XY^*Y - Y X^*Y = XY^*Y - Y \mu I_n
           = XY^*Y- (\text{Re}u) Y.
  $$
  Thus $XY^*Y = aX + bY$  for some $a, b \in \bbr $. Then
  $$
    \lambda Y^*Y = X^*(XY^*Y)
    = X^*( aX + bY) = (a \lambda + b \text{Re}\mu) I_n
  $$
  and so
  $$
    Y^*Y= \tfrac{ a \lambda + b \text{Re}\mu}{\lambda}I_n, \quad 
              \tfrac{ a \lambda + b \text{Re}\mu}{\lambda}  \in \bbr.
  $$
 Since $\mathfrak{m}'= \text{Span}_{\bbr} \{ \hat{X}, \hat{Y}\}$ is 2-dimensional, $Y$ is not a zero matrix and so from Lemma \ref{lambda}, 
 $Y \in U_{m,n}(\bbc)$.

\medskip
Conversely, assume the necessary part holds.

Assume  $\text{Im}\mu = 0$ and
$Y^{*}Y = \eta I_n$ for some $\eta > 0.$ 
Then, 
  $$
  [\hat{X}, \hat{Y}] =
  \left[
          \begin{array}{cc}
            O_n & 0 \\ 0 & M
          \end{array}
  \right]\ , \,\,
  [[\hat{X}, \hat{Y}], \hat{X}] = \widehat{MX} \,\,
  \text{ and } \,\,
  [[\hat{Y}, \hat{X}], \hat{Y}] = -\widehat{MY},
  $$  
  where $M = X Y^* - Y X^*$. 
  Note it suffices to show that 
  $[[\hat{X}, \hat{Y}], \hat{X}] \in \frakm'$  and
   $[[\hat{Y}, \hat{X}], \hat{Y}] \in \frakm'$.
  Since
  $$
    MX =  XY^*X - Y X^*X = X \bar{\mu} I_n - Y \lambda I_n
          = \text{Re}\mu X - \lambda Y,
  $$ \noindent  we get  $[[\hat{X}, \hat{Y}], \hat{X}] \in \frakm'$.
  We also get $[[\hat{Y}, \hat{X}], \hat{Y}] \in \frakm'$ since
  $$
     MY =  XY^*Y - Y X^*Y
           = X \eta I_n - Y \mu I_n
           = \eta X - \text{Re}\mu Y.
  $$

 Assume $ \text{Im}\,{\mu} \not= 0$ and $iX \in \text{Span}_{\bbr}\{ X, Y \}.$
 Since 
 $$ \text{Span}_{\bbr}\{ X, Y \} =  \text{Span}_{\bbr}\{ X, iX \}$$
 and
 $$
   \frakm' = \text{Span}_{\bbr} \{ \hat{X}, \hat{Y}\} 
           = \text{Span}_{\bbr} \{ \hat{X}, \widehat{iX}\} ,
 $$
 it suffices to show that $[[\hat{X}, \widehat{iX}], \hat{X}] \in \frakm'$  and
 $[[\widehat{iX}, \hat{X}], \widehat{iX}] \in \frakm'.$ 
 From $$h(-iv,w) = i v^{*}w = h (v,iw)$$ and from Lemma \ref{calculation},
 it is easily obtained that
 $$
   [[\hat{X},\widehat{iX}],\hat{X}]    =  -4 \lambda \widehat{iX}
 $$
 and that 
 $$
   [[\widehat{iX},\hat{X}],\widehat{iX}]   = - 4\lambda\hat{X} .
 $$
\end{proof}

\bigskip
\begin{corollary} \label{cor}
Given $X \in U_{m,n}(\bbc)$ and $Y \in M_{m,n}(\bbc)$ with
$X^* \, Y = \mu I_n$ 
for some  $\mu \in \bbc,$
and given the natural fibration
$
 U(n) \times U(m) \ra U(n,m) \ra D_{n,m},
$
assume $\frakm' = \text{Span}_{\bbr} \{ \hat{X}, \hat{Y}\} $ produce a 2-dimensional subspace of $\frakm\subset {\mathfrak u}(n,m).$ 
If $\frakm'$ gives rise to a complete totally geodesic  surface $S$ in 
$D_{n,m},$ then  $Y \in U_{m,n}(\bbc)$ and
$\mu$ determines whether $S$ is a complex submanifold or not.
\end{corollary}

\medskip
\begin{corollary}
Given $X,Y \in U_{m,n}(\bbc)$ and given the natural fibration
$
 U(n) \times U(m) \ra U(n,m) \ra D_{n,m},
$
assume $\frakm' = \text{Span}_{\bbr} \{ \hat{X}, \hat{Y}\} $ produce a 2-dimensional subspace of $\frakm\subset {\mathfrak u}(n,m).$ 
If $X^* \, Y = \mu I_n \text{ for some } \mu \in \bbr,$
then 
$\frakm'$ will give rise to a complete totally geodesic  surface $S$ in 
$D_{n,m}$ 
\end{corollary}

\medskip
\begin{remark}
Note that 
$\bbc^n \hookrightarrow \bbc^n \times \{0\} \subset \bbc^{n+m}$
is an $n$-dimensional subspace such that 
$F(v,v) \le 0$ for every $v \in \bbc^n$
and that  
$\bbc^m \hookrightarrow \{0\} \times \bbc^m   \subset \bbc^{n+m}$
is an $m$-dimensional subspace such that 
$F(w,w) \ge 0$ for every $w \in \bbc^m.$
Given $X \in U_{m,n}(\bbc), $ if $n \le m, $ then 
$X: (\bbc^n,  h_{\bbc^n}) \ra (\bbc^m,  h_{\bbc^m})$ 
is a conformal one-one linear map, 
where $h_{\bbc^k}$ is an Hermitian on
$\bbc^k, \: k= 1,2, \cdots$, given by
$$
   h_{\bbc^k}(u_1,u_2) = u_1^* u_2 \quad \text{ for } u_1, u_2 \in \bbc^k.
$$
In view of
$\hat{X} \in {\mathfrak u}(n,m) \subset \text{End}(\bbc^{n+m})$,
$\hat{X}$ sends the subspace $\bbc^n$ to the subspace $\bbc^m$
and satisfies
$$F(\hat{X}v_1, \hat{X}v_2) 
   = - \lambda  \: F(v_1, v_2)
$$
for $ v_1, v_2 \in \bbc^n.$
And the condition of the relation between $X$ and $Y$ in Theorem ~\ref{easy} says that
$$
   F(\hat{X}v_1, \hat{Y}v_2)
   = h_{\bbc^m}(Xv_1, Yv_2) = \mu \: h_{\bbc^n}(v_1,v_2)
   = - \mu \: F(v_1, v_2) \qquad
$$
for $ v_1, v_2 \in \bbc^n.$
\end{remark}

\bigskip

 When $n=1$, the condition (\ref{STAR}) is satisfied automatically for any two vectors in $\bbc^m$ by identifying
$M_{m,1}(\bbc)$ with $\bbc^m$.
So we get

\begin{corollary}
\label{geod-cond-cpn}
A 2-dimensional subspace $\frakm'$ of
$\frakm\subset {\mathfrak u}(1,m)$
gives rise to a complete totally geodesic submanifold in the affine symmetric space
${\bbc}H^m = U(1,m)/ \left(U(1) \times U(m) \right)$
if and only if 
$\frakm'$ has two linearly independent tangent vectors $\hat{w}_1$ and $\hat{w}_2$ such that
either $w_2 = i w_1$
or
$\emph{Im}h_{\bbc^m}(w_1, w_2) =0$
holds.
\end{corollary}

\bigskip
We return to the bundle
$U(n) \ra U(n,m)/U(m) \stackrel{\pi}\lra D_{n,m}$.
Any submanifold $A \subset D_{n,m}$ induces a bundle
$U(n) \ra \pi^{-1}(A) \ra A$,
which is immersed in the original bundle and diffeomorphic to the pullback bundle with respect to the inclusion of $A$ into $D_{n,m}$. In fact, in the bundle
$U(n)\times U(m) \ra U(n,m) \stackrel{\tilde{\pi}}\lra D_{n,m}$, the induced distribution in $\tilde{\pi}^{-1}(A)$
from ${\mathfrak u}(m)$ in $U(n,m)$ is integrable and preserved by the right multiplication of $U(n),$ and produces the bundle $U(n) \ra \pi^{-1}(A) \ra A$.

\begin{theorem}
\label{geod-cond-sphere1}
Assume that a complete totally geodesic  surface $S$ in $D_{n,m}$ is induced by a 2-dimensional subspace $\mathfrak{m}' \subset \mathfrak{m}$ with the necessary condition in Theorem ~\ref{easy} satisfied.
In case of 
$\emph{Im}\mu = 0,$ 
the bundle
$U(n) \ra \pi^{-1}(S) \ra S$,
which is immersed in the original bundle
$U(n) \ra U(n,m)/U(m) \stackrel{\pi}\lra D_{n,m}$,
is flat.
\end{theorem}

\begin{proof}
By a left translation, without loss of generality, assume that $S$ passes through the origin of the affine symmetric space $D_{n,m}.$

Consider the bundle
$U(n)\times U(m) \ra U(n,m) \stackrel{\tilde{\pi}}\lra D_{n,m}.$
Then $S$ induces a bundle
$U(n)\times U(m) \ra \tilde{\pi}^{-1}(S) \ra S$.
Totally geodesic condition says that the distribution induced from
$\text{Span}_{\bbr} \{ \hat{X}, \hat{Y}, [\hat{X},\hat{Y}] \}$ is integrable.
Since $\text{Im}\mu=0$ implies that $[\hat{X},\hat{Y}]$  is contained in the Lie algebra   $ {\mathfrak u}(m) $ of $U(m)$ from the equation (\ref{su_M}) in the proof of Theorem \ref{easy},  the conclusion is obtained.
\end{proof}

\bigskip

\begin{theorem}
\label{easy_generalizion}
\label{geod-cond-sphere}
Given $X \in U_{m,n}(\bbc)$ and the natural fibration
$
 U(n) \times U(m) \ra U(n,m) \stackrel{\tilde{\pi}}\lra D_{n,m},
$
consider the 2-dimensional subspace $\frakm' = \text{Span}_{\bbr} \{ \hat{X}, \widehat{iX}\} $. Then,
\begin{enumerate}
\item
$\frakm'$ gives rise to a complete totally geodesic  surface $S$ in 
$D_{n,m},$
\item
$\frakm'$ induces a $U(1)$-subbundle of a bundle 
$$U(n) \times U(m) \ra \tilde{\pi}^{-1}(S) \ra S,$$
which is an immersion of the bundle
$$
  S\big( U(1) \times U(1) \big)
  \ra
  SU(1,1)
  \ra
  SU(1,1) / S\big(U(1) \times U(1))
$$
into 
$$U(n) \times U(m) \ra U(n,m) \stackrel{\tilde{\pi}}\lra D_{n,m},$$
such that it is isomorphic to
the Hopf bundle $S^1 \ra S^{2,1} \ra \bbc H^1,$
\item
the immersion is conformal. In fact, 
$$ \big| \tilde{f}_{*}v \big| = \sqrt{n} \, |v|$$ 
under the expression $\tilde{f}: SU(1,1) \ra U(n,m)$ for the immersion.
\end{enumerate}
\end{theorem}

\begin{proof}
From Lemma \ref{lambda}, let $X^* \, X = \lambda I_n$ for some $\lambda>0$.

By a left translation, without loss of generality, assume that $S$ passes through the origin of the affine symmetric space $D_{n,m}.$

Note that,
for
$
  K =
  \left[
          \begin{array}{cc}
            - i \lambda I_n & 0 \\ 0 & iXX^{*}
          \end{array}
  \right]\
  \!\in {\mathfrak {u}}(n) \times {\mathfrak {u}}(m),
$
$$
  [\hat{X},\widehat{iX}]= -2 K,  \quad
  [K,\hat{X}]=2 \lambda \widehat{iX},  \quad
  [K,\widehat{iX}]=-2 \lambda \hat{X},
$$
which implies $[[\frakm',\frakm'],\frakm'] \subset \frakm'$ and the conclusion (1). 

Consider an orthonormal basis of
 ${\mathfrak {su}}(1,1)$:
$$
  E_1 =
  \left[
          \begin{array}{cc}
            0 & 1 \\ 1 & 0
          \end{array}
  \right]\ ,
  \quad
  E_2 =
  \left[
          \begin{array}{cc}
            0 & -i \\ i & 0
          \end{array}
  \right]\ ,
  \quad
  E_3 =
  \left[
          \begin{array}{cc}
            -i & 0 \\ 0 & i
          \end{array}
  \right]\ ,
$$
and a Lie algbra monomorphism
$f : {\mathfrak su}(1,1) \ra {\mathfrak u}(n,m)$, given by
$$
  f(aE_1+bE_2+cE_3)
  = \frac{a}{\sqrt{\lambda}} \hat{X} 
    + \frac{b}{\sqrt{\lambda}}  \widehat{iX} 
    + \frac{c}{\lambda}  K  
$$
\noindent
for $a,b,c \in \bbr,$ from
$$[E_1,E_2]= -2 E_3, \quad [E_3,E_1]=2E_2, \quad [E_3,E_2]=-2E_1.$$
For any $\theta \in \bbr,$ 
$$
  e^{\theta E_3} = 
  \left[
          \begin{array}{cc}
            e^{-i \theta} & 0 \\ 0 & e^{i \theta}
          \end{array}
  \right]\
  \in S \big( U(1) \times U(1) \big).  
$$ 
Thus $f$ will induce a Lie group monomorphism
$\tilde{f} : SU(1,1) \ra U(n,m)$ with
$
  \tilde{f}\Big({S \big( U(1) \times U(1) \big)}\Big)
  \subset
  U(n) \times U(m)
$
since $SU(1,1)$ is simply connected and
$S\big(U(1) \times U(1)\big)$ is connected.
Furthermore, it is the bundle map from
$$
  S\big( U(1) \times U(1) \big)
  \ra
  SU(1,1)
  \ra
  SU(1,1) / S\big(U(1) \times U(1))
$$
to
$$U(n) \times U(m) \ra U(n,m) \stackrel{\tilde{\pi}}\ra D_{n,m},$$
so the connected component of the integral manifold of the distribution induced by 
$\text{Span}_{\bbr}\{K,\hat{X}, \widehat{iX}\},$
which is the image of $\tilde{f},$
shows (2).

Note that 
$
 \{
   \frac{1}{\sqrt{\lambda}} \hat{X} , 
   \frac{1}{\sqrt{\lambda}}  \widehat{iX} ,
   \frac{1}{\lambda}  K  
 \}
$
is an orthogonal basis of the image of $\tilde{f}$ such that
$$
  \sqrt{n} 
  = \Big| \tfrac{1}{\sqrt{\lambda}} \hat{X} \Big| 
  = \Big| \tfrac{1}{\sqrt{\lambda}}  \widehat{iX} \Big|
  = \Big| \tfrac{1}{\lambda}  K  \Big|,
$$
which shows (3).
\end{proof}

\bigskip
\begin{remark} \label{fiber}
 Let $\hat{\theta} = \tfrac{\theta}{\lambda}.$
 Then, for $\Phi = -E_3,$
  $$
   \tilde{f}(e^{\theta \Phi})
   = \tilde{f}(e^{-\theta E_3}) 
   = e^{-\hat{\theta} K}
   =
   \left[
          \begin{array}{cc}
            e^{i \theta} I_n & 0 \\ 
            0 & I_m + \tfrac{e^{-i \theta} -1}{\lambda} XX^{*}
          \end{array}
   \right]\ 
 $$
 from
 $$
    (-i \hat{\theta} XX^{*})^{j}
    = \Big(\tfrac{-i\theta}{\lambda}\Big)^{j} X(X^{*}X)^{j-1}X^{*}
    = \tfrac{(-i \theta)^{j}}{\lambda} XX^{*}
 $$
 for $j=1,2, \cdots .$
 Furthermore,
 \begin{align*}
   &\Big(I_m + \tfrac{e^{-i \theta} -1}{\lambda} XX^{*}\Big)
    \Big(I_m + \tfrac{e^{-i \phi} -1}{\lambda} XX^{*}\Big)  \\
   &=
    I_m + \tfrac{e^{-i \theta} + e^{-i \phi} -2}{\lambda} XX^{*}
    + \tfrac{e^{-i(\theta +\phi)}-e^{-i \theta}-e^{-i \phi} +1}{\lambda ^2}
       X(X^{*}X)X^{*} \\
   &= 
    I_m + \tfrac{e^{-i(\theta +\phi)} -1}{\lambda} XX^{*},    
 \end{align*}
 from which it is also obtained that
 $$
    I_m =
    \Big(I_m + \tfrac{e^{-i \theta} -1}{\lambda} XX^{*}\Big)
    \Big(I_m + \tfrac{e^{-i \theta} -1}{\lambda} XX^{*}\Big)^{*} .
 $$
\end{remark}

\bigskip
We return to the bundle
$U(n) \ra U(n,m)/U(m) \stackrel{\pi}\lra D_{n,m}$.
In fact, Remark \ref{fiber} implies that the immersed $U(1)$-subbundle, which is the image of $\tilde{f},$ gives two $U(1)$-bundles, one of which is an immersed $U(1)$-subbundle in the bundle
$U(n) \ra U(n,m)/U(m) \stackrel{\pi}\lra D_{n,m}$
and the other one is an immersed $U(1)$-subbundle in the bundle
$U(m) \ra U(n,m)/U(n) \stackrel{\hat{\pi}}\lra D_{n,m}.$

\bigskip
\begin{theorem}
\label{thm-sphere}
Assume the same condition for  a complete totally geodesic  surface
$S$  of Theorem \ref{easy}  
and consider the immersed bundle
$U(n) \ra \pi^{-1}(S) \stackrel{\pi}{\rightarrow} S$
in the bundle
$U(n) \ra U(n,m)/U(m) \stackrel{\pi} \ra D_{n,m}.$ 
Let $\gamma$ be a piecewise smooth, simple closed curve on $S$.
Then the holonomy displacement along $\gamma$,
$$\wt\gamma(1)= \wt\gamma(0) \cdot V(\gamma),$$
is given by the right action of
$$
  V(\gamma)=e^{i \theta} I_n \quad \text{or} \quad e^{0i} I_n \in U(n),
$$
depending on whether S is a complex submanifold or not,
where $A(\gamma)$ is the area of the region on the surface $S$ surrounded
by $\gamma$ and $\theta= 2 \cdot \tfrac{1}{n} A(\gamma) .$
Especially, $\theta = 2\cdot A(\gamma)$ in case of $n=1.$
\end{theorem}

\begin{proof}
If $S$ is not a complex manifold, then, from Theorem \ref{geod-cond-sphere1}, the immersed bundle is flat, and so it is obvious that the holonomy displacement is trivial. \\
\indent
If $S$ is a complex manifold, then assume  
the condition of Theorem ~\ref{easy_generalizion} for the immersed $U(1)$-subbundle, which is the image of $\tilde{f}.$
Consider the induced map $\hat{f}: B \ra S \subset D_{n,m}$ between base spaces from the bundle map $\tilde{f}: SU(1,1) \ra \text{Im}(\tilde{f}) \subset U(n,m),$ which is a monomorphism, where $B = SU(1,1)/S(U(1) \times U(1)).$
Let 
$
 \theta 
 = 2\cdot \tfrac{1}{n} A(\gamma). 
$ 
Without loss of generality, assume that the origin of $D_{n,m}$ lies on $S$ and is the initial point of $\gamma.$
The Theorem ~\ref{easy_generalizion}, Theorem \ref{simple} and 
Remark \ref{fiber} say that
the holonomy displacement of $\gamma$ in the bundle 
$U(n) \times U(m) \ra \tilde{\pi}^{-1}(S) \stackrel{\tilde{\pi}}{\rightarrow} S,$
which is immersed in the bundle
$U(n) \times U(m) \ra U(n,m) \stackrel{\tilde{\pi}} \ra D_{n,m},$
is given by the right action of
\begin{align*}
  V(\gamma) 
  &= \tilde{f} \big( V(\hat{f}^{-1} \circ \gamma)  \big) \\
  &= \tilde{f} 
       \big( e^{2 \cdot A(\hat{f}^{-1} \circ \gamma) \Phi} \big)  \\
  &= \tilde{f} \big( e^{\theta \Phi} \big)  \\
  &=    \left[
          \begin{array}{cc}
            e^{i \theta} I_n & 0 \\ 
            0 & I_m + \tfrac{e^{-i \theta} -1}{\lambda} XX^{*}
          \end{array}
        \right]\ .
\end{align*}
Thus in the bundle 
$U(n) \ra \pi^{-1}(S) \stackrel{\pi}{\rightarrow} S,$ which is immersed in the bundle 
$U(n) \ra U(n,m)/U(m) \stackrel{\pi} \ra D_{n,m},$
the holonomy displacement is given by the right action of
$$
  V(\gamma) = e^{i \theta} I_n .
$$
\end{proof}

\bigskip
\begin{remark}
For $n=1$, we have the following Hopf bundle
 $S^1\ra S^{2m,1} \ra \bbc H^m $ under the identification
 $$S^1 = U(1), \quad \bbc H^m = D_{1,m}$$
and 
 \begin{align*}
    S^{2m,1} 
    &= U(1,m)/U(m) \\
    &\cong H^{1,2m} 
    =\{
       (z_0, \cdots , z_m) \in \bbc^{m+1} : 
       -|z_0|^2 + \sum_{k=1}^{m} |z_k|^2 = -1
     \},
 \end{align*}
 where
 $\bbc H^m = U(1,m)/(U(1) \times U(m))$ is given by the quotient metric, so the projection is a Riemannian submersion.
Let $S$ be a complete totally geodesic
surface in $\bbc H^m$ 
and  $\gamma$ be a piecewise smooth, simple closed curve on $S$.
Identify $\bbc ^m \cong M_{m,1}(\bbc).$
If $S$ is induced by $\text{Span} \{ v, w \} \subset \bbc ^m $ with
$\emph{Im}h_{\bbc ^m}(v,w)=0,$ then the holonomy displacement along $\gamma$ is trivial.
See Corollary \ref{geod-cond-cpn} and Theorem \ref{geod-cond-sphere1}.
If $S$ is induced by a two dimensional subspace with complex structure in
$\bbc ^m,$
then the holonomy displacement depends only on the area of the region surrounded by $\gamma.$ In case of $m=1,$  $\bbc H^m$ with the quotient metric is isometric to 
$\bbh^{2} \Big( \tfrac{1}{2} \Big),$ 
where
$$
 \bbh^{m}(r) 
 := \{(x_0,\cdots,x_m) \in \bbr^{m+1} : -x_0^2 + x_1^2 + \cdots + x_m^2 = - r^2\}
$$
for $m \in \bbn$ and for $r>0$.
Refer to the map $h$ defined in Section \ref{pre} and notice that the immersion $\tilde{f}$ of the Theorem \ref{easy_generalizion} is isometric if $n=1.$
\end{remark}

\medskip
\begin{remark}
Let $U(m) \ra U(n,m)/U(n) \stackrel{\hat{\pi}} \ra D_{n,m}$ be the natural
fibration.
Assume the same condition for  a complete totally geodesic  surface
$S$  of  Theorem ~\ref{easy_generalizion},
and consider the bundle
$U(m) \ra \hat{\pi}^{-1}(S) \stackrel{\hat{\pi}}{\rightarrow} S$. Let $\gamma$ be a piecewise smooth, simple closed curve on $S$.
Then the holonomy displacement along $\gamma$ is given by the right action of
$$
V(\gamma)=I_m + \tfrac{e^{-i \theta} -1}{\lambda} XX^{*} \in U(m),
$$
which depends on $X,$ not only on $n,$
where $\theta= 2 \cdot \tfrac{1}{n} A(\gamma).$
\end{remark}

\bibliographystyle{amsplain}

\end{document}